\documentclass[11pt]{amsart}
\usepackage{amsmath,amssymb,mathrsfs,enumerate,hyperref}
\usepackage[alphabetic,initials]{amsrefs}
\newtheorem{theorem}{Theorem}[section]
\newtheorem{lemma}[theorem]{Lemma}
\newtheorem{proposition}[theorem]{Proposition}
\newtheorem{corollary}[theorem]{Corollary}
\theoremstyle{definition}

\newtheorem{remark}[theorem]{Remark}
\newtheorem*{acknowledgement}{Acknowledgement}

\begin{document}
\title[Area-preserving maps under coupled flow]{Deforming area-preserving maps between surfaces by mean curvature flow coupled with Ricci flow}
\author{Ping-Hung Lee}
\address{Columbia University \\ 2990 Broadway \\ New York NY 10027 \\ USA}
\email{pl2975@columbia.edu}
\begin{abstract}
We study a natural way to deform area-preserving maps between compact Riemann surfaces. Specifically, we evolve the metrics on the two Riemann surfaces by the normalized Ricci flow and the graph of the area-preserving map by the mean curvature flow. We prove that the flow exists for all time, remains the graph of an area-preserving map, and converges smoothly and exponentially to a minimal Lagrangian graph with respect to the product of the limiting metrics. This generalizes earlier results of Wang and Smoczyk, in which the Riemann surfaces have constant scalar curvature.
\end{abstract}
\maketitle

\section{Introduction}
\label{section:introduction}

Let $(\Sigma,g_0,\omega_0)$ and $(\tilde\Sigma,\tilde g_0,\tilde\omega_0)$ be two compact Riemann surfaces with the same average scalar curvature $2\kappa$. Evolve $(\Sigma,g_t,\omega_t)$ and $(\tilde\Sigma,\tilde g_t,\tilde\omega_t)$ by the normalized Ricci flow
\begin{equation}
\label{equation:RF}
\partial_tg_{ij}
=\kappa g_{ij}-\operatorname{Ric}(g)_{ij},\quad\partial_t\tilde g_{ij}
=\kappa\tilde g_{ij}-\operatorname{Ric}(\tilde g)_{ij}.
\end{equation}
It is known \cites{Ham88,Cho91} that $g_t$ and $\tilde g_t$ exist for all time $t\ge0$, and that the average scalar curvatures of $g_t$ and $\tilde g_t$ remain equal to $2\kappa$. Moreover, the metrics converge smoothly and exponentially to the metrics $g_\infty$ and $\tilde g_\infty$ of constant scalar curvature $2\kappa$ on $\Sigma$ and $\tilde\Sigma$; that is, there exists $\mu>0$ with
\begin{equation}
\label{equation:g.convergence}
\|g_t-g_\infty\|_{C^k}+\|\partial_t^\ell g_t\|_{C^k}+\|\tilde g_t-\tilde g_\infty\|_{C^k}+\|\partial_t^\ell\tilde g_t\|_{C^k}
\le C_{k,\ell}e^{-\mu t}
\end{equation}
for all $k,\ell>0$ and $t\ge0$.

Let $M=\Sigma\times\tilde\Sigma$, which admits a product metric $G_t=g_t\oplus\tilde g_t$. Then, $(M,G_t)$ satisfies the normalized Ricci flow equation, and $G_t\to G_\infty=g_\infty\oplus\tilde g_\infty$ exponentially fast. Moreover, $\omega_t'=\omega_t-\tilde\omega_t$ and $\omega_t''=\omega_t+\tilde\omega_t$ are two K\"ahler forms on $(M,G_t)$. Let $J$ denote the time-independent complex structure associated with $(M,G_t,\omega_t')$.

Let $f_0:(\Sigma,g_0)\to(\tilde\Sigma,\tilde g_0)$ be an area-preserving map; thus, $f_0^*\tilde\omega_0=\omega_0$. Then, the embedding $F_0:\Sigma\to M,\,x\mapsto(x,f_0(x))$ is Lagrangian with respect to $\omega'_0$. Evolve $F_t:\Sigma\to M$ by the mean curvature flow
\begin{equation}
\label{equation:MCF}
\partial_tF
=H,
\end{equation}
for $t\in[0,T)$, where $T\in(0,\infty]$ denotes the maximal existence time. Denote the image of $F_t$ by $\Gamma_t$. As discovered in \cite{Smo99}, the Lagrangian condition on $\Gamma_t$ is preserved under the coupled K\"ahler--Ricci flow \eqref{equation:RF} and mean curvature flow \eqref{equation:MCF}. 
Our main result is as follows.

\begin{theorem}
\label{theorem:main}
Let $(\Sigma,g_t,\omega_t)$ and $(\tilde\Sigma,\tilde g_t,\tilde\omega_t)$ be compact Riemann surfaces evolving by the normalized Ricci flow \eqref{equation:RF}, with the same average scalar curvature $2\kappa$. Let $\Gamma_t\subset(\Sigma\times\tilde\Sigma,\omega_t-\tilde\omega_t)$ be evolving by the mean curvature flow \eqref{equation:MCF}, with $\Gamma_0=\mathrm{graph}(f_0)$ for an area-preserving map $f_0:(\Sigma,g_0)\to(\tilde\Sigma,\tilde g_0)$. Then the flow exists for all time and $\Gamma_t=\operatorname{graph}(f_t)$ for an area-preserving map $f_t:(\Sigma,g_t)\to(\tilde\Sigma,\tilde g_t)$ for all $t\ge0$. Moreover, $\Gamma_t$ converges smoothly and exponentially as $t\to\infty$ to a minimal surface $\Gamma_\infty=\operatorname{graph}(f_\infty)\subset(\Sigma,g_\infty)\times(\tilde\Sigma,\tilde g_\infty)$ for an area-preserving map $f_\infty:(\Sigma,g_\infty)\to(\tilde\Sigma,\tilde g_\infty)$. If $\kappa\ge0$, then $\Gamma_\infty$ is totally geodesic; if $\kappa>0$, then $f_\infty$ is an isometry.
\end{theorem}

This theorem extends previous works by Wang \cites{Wan01b,Wan08} and Smoczyk \cite{Smo02}, where the two Riemann surfaces are assumed to have the same constant curvature and thus are stationary under Ricci flow. Han, Li, and Zhao \cite{HLZ18} also studied K\"ahler--Ricci mean curvature flow and proved long-time existence and convergence under suitable graphical and symplectic assumptions. In higher dimensions, such a coupled flow has been considered by Lee, Tam, and Wan \cite{LTW25}.

The paper is organized as follows. Section~\ref{section:preliminaries} collects notation, evolution equations, and basic facts from \cites{Wan01a,Wan01b,HLZ18}. Section~\ref{section:longtime} proves long-time existence by White's regularity theorem. Section~\ref{section:integral} derives an integral estimate for $\int|H|^2/\eta$. Section~\ref{section:A.bound} establishes a time-independent bound on the second fundamental form $A$. Section~\ref{section:convergence} proves the exponential convergence of the flow and Theorem~\ref{theorem:main}.

The arguments of this paper are similar to those in \cites{Wan01a,Wan01b,Wan08}, with two main modifications. In Section~\ref{section:A.bound}, we give a new proof of the time-independent bound on $|A|$ when $\kappa<0$. In Section~\ref{section:convergence}, we improve the smooth convergence result in \cite{Wan08} to exponential convergence using energy estimates.

{\it Conventions}. Throughout this work, all constants $C$ depend on the initial data $G_0$ and $F_0$, with any additional dependencies indicated explicitly by subscripts. The value of $C$ may vary from one occurrence to another.

\begin{acknowledgement}
The author is grateful to his advisor, \mbox{Simon Brendle}, for invaluable guidance and constant encouragement. The author also thanks \mbox{Tang-Kai Lee} and \mbox{Mu-Tao Wang} for comments on an earlier draft. 
\end{acknowledgement}

\section{Preliminaries}
\label{section:preliminaries}

We record the notation and preliminary results, mostly from \cites{Wan01a,Wan01b,HLZ18}.

\subsection{Notation and evolution of the area form}

Let $R_{ABCD}$, $\operatorname{Ric}_{AB}$, and $R_M$ denote the components of the Riemann curvature tensor, Ricci curvature, and scalar curvature of $(M,G_t)$, respectively; we use the sign convention for which $R_{ABAB}$ is the sectional curvature. Let $R_\Sigma$ and $R_{\tilde\Sigma}$ denote the scalar curvatures of $(\Sigma,g_t)$ and $(\tilde\Sigma,\tilde g_t)$, respectively.

Pick a time-independent local coordinate system $(x^i)_{i=1,2}$ on $\Sigma$. Let $F(p,t)=F_t(p)$ and let $\mathcal N\subset F^*TM$ denote the normal bundle along the flow; here, the fiber $T_{F_t(p)}M$ of $F^*TM$ at $(p,t)$ is endowed with the metric $G_t$. Pick a time-dependent orthonormal local frame $(e_\alpha)_{\alpha=3,4}$ of $\mathcal N$. We use the convention that Latin indices $i,j,k,l,m$ refer to $(x^i)$ and Greek indices $\alpha,\beta,\gamma$ to $(e_\alpha)$. Write $F_i=\frac{\partial F}{\partial x^i}$.

The components of the second fundamental form $A$ of $\Gamma_t$ are defined by $h_{ij}^\alpha=G_t(\overline\nabla_{F_i}F_j,e_\alpha)$. Let $G^{ij}$ denote the inverse of $G_{ij}=G_t(F_i,F_j)$. Then, the mean curvature vector is defined as $H^\alpha e_\alpha$, where $H^\alpha=G^{ij}h_{ij}^\alpha$.

We recall some evolution equations from \cite{HLZ18}. Along the coupled flow, the induced metric and the area form on $\Gamma_t$ evolve by
\begin{align}
\label{equation:G.evolution}
\partial_tG_{ij}
&=-2H^\alpha h_{ij}^\alpha-\operatorname{Ric}_{ij}+\kappa G_{ij},\\
\label{equation:area.evolution}
\partial_td\mathrm{area}
&=(-|H|^2-\frac12G^{ij}\operatorname{Ric}_{ij}+\kappa)d\mathrm{area}.
\end{align}
In particular, because $\lvert-\frac12G^{ij}\operatorname{Ric}_{ij}+\kappa\rvert\le Ce^{-\mu t}$ by \eqref{equation:g.convergence}, we obtain
\[
\frac d{dt}\left(e^{Ce^{-\mu t}}\mathrm{area}(\Gamma_t)\right)
\le-e^{Ce^{-\mu t}}\int_{\Gamma_t}|H|^2
\le0
\]
for $t\in[0,T)$. Therefore,
\begin{equation}
\label{equation:area.bound}
\mathrm{area}(\Gamma_t)
\le C\mathrm{area}(\Gamma_0)
\end{equation}
for $t\in[0,T)$ and
\begin{equation}
\label{equation:H.L2}
\int_0^T\int_{\Gamma_t}|H|^2
\le C\mathrm{area}(\Gamma_0).
\end{equation}

\subsection{Algebraic properties of $A$}

The second fundamental form $A$ of a Lagrangian surface enjoys an additional symmetry, yielding the following inequalities. The first one can be found in \cite{Wan01b}.

\begin{lemma}
\label{lem:A.symmetry}
For a Lagrangian surface in a K\"ahler manifold,
\begin{enumerate}[(i)]
\item\label{lem:A.symmetry1} $|H|^2\le\frac43|A|^2$;

\item\label{lem:A.symmetry2} $\bigl|2|H^\alpha h_{ij}^\alpha|^2-|A|^2|H|^2\bigr|\le\frac2{\sqrt3}|A|\,|H|^3$.
\end{enumerate}
\end{lemma}

\begin{proof}
Fix $p\in\Gamma$ and an orthonormal frame $(e_1,e_2)$ of $T_p\Gamma$; by the Lagrangian condition, $(Je_1,Je_2)$ is an orthonormal frame of the normal space. Write $h_{ijk}=G_t(\overline\nabla_{e_i}e_j,Je_k)$. Then $h_{ijk}=h_{jik}$, and
\[
h_{ijk}
=-G_t(e_j,J\overline\nabla_{e_i}e_k)
=G_t(Je_j,\overline\nabla_{e_i}e_k)
=h_{ikj}
\]
by $\overline\nabla J=0$. Hence, $h_{ijk}$ is totally symmetric in $i,j,k$.

Both claims are trivial when $H=0$, so we assume $H\ne0$ and choose the frame such that $Je_1=H/|H|$. Write $a=h_{111}$ and $b=h_{112}$. The total symmetry and the trace conditions $h_{ii1}=|H|$ and $h_{ii2}=0$ imply
\[
h_{122}
=h_{212}
=h_{221}
=|H|-a,\quad h_{112}
=h_{121}
=h_{211}
=b,\quad h_{222}
=-b.
\]
Let $u=2(a-\tfrac34|H|)$ and $v=\frac{\sqrt3}2|H|$. Then
\begin{equation}
\label{equation:A.complete.square}
|A|^2
=a^2+3(|H|-a)^2+4b^2
=u^2+v^2+4b^2.
\end{equation}
In particular, $v^2\le|A|^2$, which proves (i). Next, $H^\alpha h_{ij}^\alpha=|H|h_{ij1}$, so
\[
2|H^\alpha h_{ij}^\alpha|^2-|A|^2|H|^2
=|H|^2\left(a^2-(|H|-a)^2\right)
=|H|^3(u+\frac1{\sqrt3}v).
\]
The Cauchy--Schwarz inequality and \eqref{equation:A.complete.square} then give (ii).
\end{proof}

\subsection{Graphical quantity $\eta$ and its evolution}

A useful quantity introduced by Wang \cite{Wan01a} to study graphical mean curvature flow is $\eta_t=*(\omega_t''|_{\Gamma_t})\in[-1,1]$, which equals $*(2\omega_t|_{\Gamma_t})$ by the Lagrangian condition $\omega_t'|_{\Gamma_t}=0$. By direct computation, $\eta_0>0$. By \cite{HLZ18}*{(2.12)} and by $|H|^2\le\frac43|A|^2$ in Lemma~\ref{lem:A.symmetry}\eqref{lem:A.symmetry1},
\begin{align}
\label{equation:eta.evolution}
(\partial_t-\Delta)\eta
&=(2|A|^2-|H|^2)\eta+\frac{R_M}4\eta(1-\eta^2)\\
&
\label{equation:eta.ev2}
\ge\frac23|A|^2\eta+\frac{R_M}4\eta(1-\eta^2).
\end{align}
Since $\lvert\frac{R_M}4-\kappa\rvert\le Ce^{-\mu t}$ by \eqref{equation:g.convergence}, we arrive at
\[
(\partial_t-\Delta)\eta
\ge(\kappa-Ce^{-\mu t})\eta(1-\eta^2).
\]
By the maximum principle, we then obtain
\begin{equation}
\label{equation:eta.bound}
\eta
\ge\frac{u(t)}{\sqrt{1+u(t)^2}}>0,
\end{equation}
where
\[
u(t)
=\frac{\eta_*}{\sqrt{1-\eta_*^2}}\exp\left(\kappa t-\frac C\mu(1-e^{-\mu t})\right)\in(0,\infty],\quad\eta_*
=\min_{\Gamma_0}\eta\in(0,1].
\]
Combining this with the Lagrangian condition on $\Gamma_t$, we obtain:

\begin{lemma}
\label{lem:eta}
\begin{enumerate}[(i)]
\item\label{lem:eta..} $\Gamma_t$ remains the graph of an area-preserving map $f_t$ for every $t\in[0,T)$, and $\eta\ge C_{t_0}^{-1}$ and $|df_t|_{\mathrm{op}}\le C_{t_0}$ on $\Sigma\times[0,t_0)$ for every finite $t_0\in(0,T]$, where $|df_t|_{\mathrm{op}}$ denotes the operator norm.

\item\label{lem:eta.>0} If $\kappa>0$ and $T=\infty$, then $\eta\to1$ uniformly as $t\to\infty$.

\item\label{lem:eta.=0} If $\kappa=0$, then $\eta$ has a time-independent lower bound $C^{-1}>0$.

\item\label{lem:eta.<0} If $\kappa<0$, then $\eta\ge C^{-1}e^{\kappa t}>0$.
\end{enumerate}
\end{lemma}

For reference, we recall from \cite{Wan01a}*{p.~334} that
\begin{equation}
\label{equation:nabla.eta.bound}
|\nabla\eta|^2
\le2(1-\eta^2)|A|^2.
\end{equation}

\subsection{Evolution of $A$ and $H$}

Recall from \cite{HLZ18} that
\begin{equation}
\label{equation:A.evolution}
(\partial_t-\Delta)(|A|^2)
\le-2|\nabla A|^2+c_1|A|+c_1|A|^4
\end{equation}
for some $c_1>0$. Combining \eqref{equation:eta.ev2} and \eqref{equation:A.evolution}, we obtain
\begin{align*}
(\partial_t-\Delta)(\eta^{-2p}|A|^2)
&=(-2p\eta^{-2p-1}(\partial_t-\Delta)\eta-2p(2p+1)\eta^{-2p-2}|\nabla\eta|^2)|A|^2\\
&+\eta^{-2p}(\partial_t-\Delta)(|A|^2)+4p\eta^{-2p-1}\nabla\eta\cdot\nabla(|A|^2)\\
&\le(-2p\eta^{-2p-1}(\frac23|A|^2\eta+\frac{R_M}4\eta(1-\eta^2))\\
&-2p(2p+1)\eta^{-2p-2}|\nabla\eta|^2)|A|^2\\
&+\eta^{-2p}(-2|\nabla A|^2+c_1|A|+c_1|A|^4)\\
&+4p\eta^{-2p-1}\nabla\eta\cdot\nabla(|A|^2),
\end{align*}
for any constant $p>0$. Observe that
\begin{align*}
|\nabla\eta\cdot\nabla(|A|^2)|
&\le|\nabla\eta|\,|\nabla(|A|^2)|\\
&\le2\sqrt{2(1-\eta^2)}|A|^2|\nabla A|
\le\sqrt{2(1-\eta^2)}(|A|^4+|\nabla A|^2),
\end{align*}
with the second inequality following from Kato's inequality $|\nabla|A||\le|\nabla A|$ and \eqref{equation:nabla.eta.bound}. It follows that
\begin{align}
\label{equation:etaA.evolution}
\eta^{2p}(\partial_t-\Delta)(\eta^{-2p}|A|^2)
&\le(-2+4p\eta^{-1}\sqrt{2(1-\eta^2)})|\nabla A|^2\\
\nonumber&-\frac{R_M}2p(1-\eta^2)|A|^2+c_1|A|\\
\nonumber&+(-\frac43p+c_1+4p\eta^{-1}\sqrt{2(1-\eta^2)})|A|^4.
\end{align}

We also need the evolution equation for $|H|^2$. To derive it, we further assume that the coordinates $x^i$ are normal with respect to $F_{t_0}^*G_{t_0}$ at $p_0\in\Sigma$, and write $q=F_{t_0}(p_0)=(p,\tilde p)$. In fact, there exist positively oriented orthonormal bases $a_1,a_2\in T_p(\Sigma,g_{t_0})$ and $b_1,b_2\in T_{\tilde p}(\tilde\Sigma,\tilde g_{t_0})$ with $df_{t_0}(a_i)=\lambda_ib_i$ for some $\lambda_1,\lambda_2\ge0$ with $\lambda_1\lambda_2=\det(df_{t_0})=1$; hence, we may pick
\[
F_i
=\frac1{\sqrt{1+\lambda_i^2}}(a_i,\lambda_ib_i),\quad e_\alpha
=\frac1{\sqrt{1+\lambda_{\alpha-2}^2}}(-\lambda_{\alpha-2}a_{\alpha-2},b_{\alpha-2})
\]
for $i=1,2$ and $\alpha=3,4$ at $(q,t_0)$. Below, we perform the computation at $(q,t_0)$. By direct computation,
\begin{gather}
\nonumber\eta
=\omega''(F_1,F_2)
=\frac{1+\lambda_1\lambda_2}{\sqrt{(1+\lambda_1^2)(1+\lambda_2^2)}},\\
\label{equation:Rakbk}
R_{\alpha k\beta k}
=
\begin{cases}
0,
&\alpha\ne\beta\\
\frac{\lambda_1^2R_\Sigma+\lambda_2^2R_{\tilde\Sigma}}{2(1+\lambda_1^2)(1+\lambda_2^2)},
&\alpha
=\beta
=3\\
\frac{\lambda_1^2R_{\tilde\Sigma}+\lambda_2^2R_\Sigma}{2(1+\lambda_1^2)(1+\lambda_2^2)},
&\alpha
=\beta
=4.
\end{cases}
\end{gather}
Using $\lambda_1\lambda_2=1$, we obtain
\begin{equation}
\label{equation:R3k3k}
R_{3k3k}-\frac\kappa2(2-\eta^2)
=\frac{\lambda_1^2(R_\Sigma-2\kappa)+\lambda_2^2(R_{\tilde\Sigma}-2\kappa)}{2(1+\lambda_1^2)(1+\lambda_2^2)}
\end{equation}
at $(q,t_0)$. The $R_{4k4k}$ term is similar.

Recall from \cite{HLZ18}*{(2.5)} that
\begin{align*}
(\partial_t-\Delta)h_{ij}^\alpha
&=\overline\nabla_kR_{\alpha ijk}+\overline\nabla_jR_{\alpha kik}\\
&-2R_{lijk}h_{lk}^\alpha+2R_{\alpha\beta jk}h_{ik}^\beta+2R_{\alpha\beta ik}h_{jk}^\beta\\
&-R_{lkik}h_{lj}^\alpha-R_{lkjk}h_{li}^\alpha+R_{\alpha k\beta k}h_{ij}^\beta\\
&-h_{im}^\alpha(H^\gamma h_{mj}^\gamma-h_{mk}^\gamma h_{jk}^\gamma)-h_{mk}^\alpha(h_{mj}^\gamma h_{ik}^\gamma-h_{mk}^\gamma h_{ij}^\gamma)\\
&-h_{ik}^\beta(h_{lj}^\beta h_{lk}^\alpha-h_{lk}^\beta h_{lj}^\alpha)-h_{jk}^\alpha h_{ik}^\beta H^\beta+h_{ij}^\beta G_t(e_\beta,\overline\nabla_He_\alpha)\\
&-\operatorname{Ric}_{\alpha\beta}h_{ij}^\beta+\kappa h_{ij}^\alpha-\frac12(\overline\nabla_iR_{j\alpha}+\overline\nabla_jR_{i\alpha}-\overline\nabla_\alpha R_{ij}),
\end{align*}
where $\overline\nabla_He_\alpha|_{(p,t)}=(F^*\overline\nabla^{G_t})_{\partial_t}e_\alpha|_{(p,t)}$ and $\overline\nabla^{G_t}$ denotes the Levi--Civita connection of $(M,G_t)$. Thus,
\begin{align}
\label{equation:H.evolution}
2G^{ij}((\partial_t-\Delta)h_{ij}^\alpha)H^\alpha
&=2(R_{\alpha k\beta k}-\operatorname{Ric}_{\alpha\beta})H^\alpha H^\beta\\
\nonumber&-2|H^\alpha h^\alpha_{ij}|^2+2H^\alpha H^\beta G(e_\alpha,\overline\nabla_He_\beta)\\
\nonumber&+2\kappa|H|^2-(2\overline\nabla_iR_{i\alpha}-\overline\nabla_\alpha R_{ii})H^\alpha.
\end{align}
Meanwhile,
\[
0
=\partial_t(G(e_\alpha,e_\beta))
=(\kappa G-\operatorname{Ric})(e_\alpha,e_\beta)+G(e_\alpha,\overline\nabla_He_\beta)+G(\overline\nabla_He_\alpha,e_\beta),
\]
so
\begin{equation}
\label{equation:HHG}
2H^\alpha H^\beta G(e_\alpha,\overline\nabla_He_\beta)
=\operatorname{Ric}_{\alpha\beta}H^\alpha H^\beta-\kappa|H|^2.
\end{equation}
Now, $|H|^2=G^{ij}G^{kl}h_{ij}^\alpha h_{kl}^\alpha$. Thus,
\[
(\partial_t-\Delta)|H|^2
=2(\partial_tG^{ij})G^{kl}h_{ij}^\alpha h_{kl}^\alpha+2G^{ij}G^{kl}((\partial_t-\Delta)h_{ij}^\alpha)h_{kl}^\alpha-2|\nabla H|^2.
\]
Here, $\partial_tG^{ij}=2H^\alpha h_{ij}^\alpha+\operatorname{Ric}_{ij}-\kappa\delta_{ij}$ by \eqref{equation:G.evolution}. Combining this with \eqref{equation:H.evolution} and \eqref{equation:HHG}, we obtain
\begin{align}
\label{equation:H.evolution2}
(\partial_t-\Delta)|H|^2
&=-2|\nabla H|^2+2|H^\alpha h^\alpha_{ij}|^2+2\operatorname{Ric}_{ij}H^\alpha h^\alpha_{ij}\\
\nonumber&+(\overline\nabla_\alpha\operatorname{Ric}_{ii}-2\overline\nabla_i\operatorname{Ric}_{i\alpha})H^\alpha\\
\nonumber&-\kappa|H|^2+(2R_{\alpha k\beta k}-\operatorname{Ric}_{\alpha\beta})H^\alpha H^\beta.
\end{align}
By \eqref{equation:g.convergence}, \eqref{equation:Rakbk}, and \eqref{equation:R3k3k}, we know $\operatorname{Ric}-\kappa G$, $\overline\nabla\operatorname{Ric}$, and $2R_{\alpha k\beta k}-\kappa(2-\eta^2)G_{\alpha\beta}$ are $O(e^{-\mu t})$. Hence, \eqref{equation:H.evolution2} becomes
\begin{equation}
\label{equation:H.evolution3}
(\partial_t-\Delta)|H|^2
=-2|\nabla H|^2+2|H^\alpha h_{ij}^\alpha|^2+\kappa(2-\eta^2)|H|^2+\mathcal E_1,
\end{equation}
where
\begin{equation}
\label{ineq:E1}
|\mathcal E_1|
\le Ce^{-\mu t}|H|(1+|A|).
\end{equation}

\subsection{Gauss--Bonnet theorem}

Let $\overline K(T\Gamma_t)$ denote the sectional curvature of $T\Gamma_t$ in $(M,G_t)$. Then,
\[
\overline K(T\Gamma_t)
=\frac{R_\Sigma}2(\frac\eta2)^2+\frac{R_{\tilde\Sigma}}2(\frac\eta2)^2
=\frac18R_M\eta^2
\]
by the Lagrangian condition and the definition of $\eta$. Therefore, the Gaussian curvature of $\Gamma_t$ is
\begin{equation}
\label{equation:K}
K
=\frac18R_M\eta^2+\frac12|H|^2-\frac12|A|^2,
\end{equation}
and the Gauss--Bonnet theorem reads
\begin{equation}
\label{equation:GB}
2\pi\chi(\Gamma_t)
=\int_{\Gamma_t}K
=\frac18\int_{\Gamma_t}R_M\eta^2+\frac12\int_{\Gamma_t}|H|^2-\frac12\int_{\Gamma_t}|A|^2.
\end{equation}
Thus,
\begin{equation}
\label{ineq:A2}
\int_{\Gamma_t}|A|^2
\le C+\int_{\Gamma_t}|H|^2.
\end{equation}

\section{Long-time existence}
\label{section:longtime}

\begin{proposition}
\label{proposition:longtime}
Under the same assumptions as in Theorem~\ref{theorem:main}, the mean curvature flow \eqref{equation:MCF} exists for all time.
\end{proposition}

The proof of Proposition~\ref{proposition:longtime} occupies the remainder of this section; it follows \cite{Wan01a}*{Proposition~6.1} closely. Similar adaptations can also be found in \cite{HLZ18}*{Theorem~4.1} and \cite{LTW25}*{Theorem~2.1}.

The key idea of the proof is to apply White's regularity theorem, which states that no singularities can form at a point with Gaussian density at most $1$. To this end, we isometrically embed the evolving ambient space $(M,G_t)$ into Euclidean space. By Huisken's monotonicity formula, we prove that $\int|A|^2$ is small over rescaled surfaces $\Gamma_{s_k}^k$ to be defined later. Combining this with the graphical assumption, we show that $\Gamma_{s_k}^k$ becomes arbitrarily close to a plane. Hence, its Gaussian density is at most 1, as desired.

\subsection{Isometric embedding}

In Wang's work, the ambient space is stationary, and he embeds $(M,G_0)\to\mathbb R^N$. Our ambient space evolves under K\"ahler--Ricci flow, so we instead embed the spacetime $\mathcal M:=M\times[0,\infty)$ into $\mathbb R^N$. To ensure a uniform bound on the second fundamental form of the embedding, we use the following compactification.

Let $\widehat{\mathcal M}=M\times[0,2]$ be the product manifold endowed with the metric $G_{(1-\tau)^{-1}-1}+d\tau^2$ for $\tau\in[0,1)$ and $G_\infty+d\tau^2$ for $\tau\in[1,2]$. Write $t(\tau)=(1-\tau)^{-1}-1$. For $\tau\in(0,1)$ we compute
\[
\frac{d^jt(\tau)}{d\tau^j}
=j!\,(1-\tau)^{-j-1}
=j!\,(t(\tau)+1)^{j+1},
\]
which is a polynomial in $t(\tau)$. Since $G_t\to G_\infty$ exponentially fast, the above metric on $\widehat{\mathcal M}$ is smooth. Nash's embedding theorem for compact manifolds with boundary gives an isometric embedding $\hat\iota:\widehat{\mathcal M}\to\mathbb R^N$ for some $N>0$.

Endow $\mathcal M=M\times[0,\infty)$ with the metric $G_t+\frac1{(t+1)^4}dt^2$. Then the embedding $\mathcal M\to\widehat{\mathcal M},\ (p,t)\mapsto(p,\frac t{t+1})$ is isometric. Composing this map with $\hat\iota$, we obtain an isometric embedding $\iota:\mathcal M\to\mathbb R^N$. Then, the norm $|A_{M\times\{t\},\mathbb R^N}|$ of the second fundamental form of $\iota(M\times\{t\})$ in $\mathbb R^N$ is uniformly bounded over $t\in[0,\infty)$. The mean curvature flow equation \eqref{equation:MCF} now reads
\[
\partial_t(\iota\circ F)
=\iota_*(H+\partial_t)
=\overline H+E,
\]
where $H\in TM/T\Gamma_t$ is the mean curvature vector of $\Gamma_t$ in $M$, $\overline H\in T\mathbb R^N/T\Gamma_t$ is the mean curvature vector of $\Gamma_t$ in $\mathbb R^N$,
\[
E
=\iota_*(\partial_t)-\sum_iA_{M\times\{t\},\mathbb R^N}(e_i,e_i),
\]
and $(e_i)$ is an orthonormal frame of $T\Gamma_t$. Compared to the one in Wang's work, the expression for $E$ here contains an additional $\iota_*(\partial_t)$ term, but $E$ is still orthogonal to $T\Gamma_t$ since $\iota$ is an isometry. Also,
\begin{equation}
\label{equation:E.bound}
|E|
\le|\iota_*(\partial_t)|+\sum_i|A_{M\times\{t\},\mathbb R^N}(e_i,e_i)|
\le\frac1{(t+1)^2}+C
\le C
\end{equation}
for all $t\in[0,\infty)$.

\subsection{Huisken's monotonicity formula}

Let $x_0\in\mathbb R^N$, $\epsilon\in(0,\frac T2)$, and $t_0\in[2\epsilon,\infty)$. For $x\in\mathbb R^N$ and $t<t_0$, let
\[
\rho_{x_0,t_0}(x,t)
=\frac1{4\pi(t_0-t)}e^{-\frac{|x-x_0|^2}{4(t_0-t)}}
\]
denote a backward heat kernel centered at $(x_0,t_0)$. We then have the following form of Huisken's monotonicity formula:
\begin{align}
\label{equation:Huisken}
\frac d{dt}\int_{\Gamma_t}\psi\rho_{x_0,t_0}
=
&\int_{\Gamma_t}(\partial_t\psi-\Delta\psi)\rho_{x_0,t_0}\\
\nonumber+
&\int_{\Gamma_t}\left(-\left|\overline H+\frac{(F-x_0)^\perp}{2\left(t_0-t\right)}+\frac E2\right|^2+\frac{|E|^2}4\right)\psi\rho_{x_0,t_0}
\end{align}
for any smooth function $\psi$ on $\Sigma\times(0,T)$; see, e.g., \cite{Wan01a}*{(5.7)} for a derivation. Such a formula was originally derived by Huisken \cite{Hui90} in Euclidean space and by White \cite{Whi97} for general ambient spaces.

First, we take $\psi=1$ in \eqref{equation:Huisken}. Recall from \eqref{equation:E.bound} that $|E|^2\le C$ for $t\in[0,T)$. It follows that for $t\in[0,\min(t_0,T))$,
\[
\frac d{dt}\int_{\Gamma_t}\rho_{x_0,t_0}
\le C\int_{\Gamma_t}\rho_{x_0,t_0},
\]
i.e., $e^{-Ct}\int_{\Gamma_t}\rho_{x_0,t_0}$ is non-increasing. Since $\int_{\Gamma_t}\rho_{x_0,t_0}\ge0$, we find that the Gaussian density $\Theta_{x_0,t_0}:=\lim_{t\to t_0^-}\int_{\Gamma_t}\rho_{x_0,t_0}\in[0,\infty)$ exists provided $t_0\le T$. Moreover, for $x\in\mathbb R^N$ and $t=t_0-2\epsilon$ we have
\[
\rho_{x_0,t_0}(x,t)
=\frac1{8\pi\epsilon}e^{-\frac{|x-x_0|^2}{8\epsilon}}
\le\frac1{8\pi\epsilon}.
\]
Combining this with the area bound in \eqref{equation:area.bound}, we obtain
\[
\int_{\Gamma_{t_0-2\epsilon}}\rho_{x_0,t_0}
\le C_\epsilon
\]
provided that $t_0-2\epsilon<T$. Since $e^{-Ct}\int_{\Gamma_t}\rho_{x_0,t_0}$ is non-increasing for $t\in[t_0-2\epsilon,\min(t_0,T))$, we obtain the following estimate of the Gaussian area.

\begin{lemma}
\label{lem:gaussian.area}
For every $\epsilon>0$, there exists a (time-independent) constant $C_\epsilon>0$ such that $\int_{\Gamma_t}\rho_{x_0,t_0}\le C_\epsilon$ for every $x_0\in\mathbb R^N$, $t_0\in[2\epsilon,T+\epsilon)$, and $t\in[t_0-2\epsilon,\min(t_0,T))$.
\end{lemma}

Lemma~\ref{lem:gaussian.area} has the following standard corollary that will be used later.

\begin{corollary}
\label{corollary:area.growth}
For every $\epsilon>0$, there exists a constant $C_\epsilon>0$ such that $\mathrm{area}(\Gamma_t\cap B_r(x_0))\le C_\epsilon r^2$ for $x_0\in\mathbb R^N$, $t\in[2\epsilon,T)$, and $r\in(0,\sqrt\epsilon]$.
\end{corollary}

\begin{proof}
Let $t_0=t+r^2\in(2\epsilon,T+\epsilon)$. Then,
\[
\rho_{x_0,t_0}(x,t)
=\frac1{4\pi r^2}e^{-\frac{|x-x_0|^2}{4r^2}}
\ge\frac1{4\pi r^2}e^{-\frac14}
\]
for $x\in B_r(x_0)$. It follows that
\[
\mathrm{area}(\Gamma_t\cap B_r(x_0))
\le4\pi r^2e^{\frac14}\int_{\Gamma_t}\rho_{x_0,t_0}.
\]
The corollary now follows directly from Lemma~\ref{lem:gaussian.area}.
\end{proof}

Now, assume that the maximal existence time $T<\infty$ and take $t_0=T$ and $\psi=2-\eta$ in \eqref{equation:Huisken}. Note that $1\le\psi\le2$. Using \eqref{equation:eta.ev2}, we find
\begin{align*}
\frac d{dt}\int_{\Gamma_t}(2-\eta)\rho_{x_0,T}
&\le-\int_{\Gamma_t}(\frac23|A|^2\eta+\frac{R_M}4\eta(1-\eta^2))\rho_{x_0,T}\\
&+C\int_{\Gamma_t}(2-\eta)\rho_{x_0,T}\\
&\le-\frac1{C_T}\int_{\Gamma_t}|A|^2\rho_{x_0,T}+C\int_{\Gamma_t}(2-\eta)\rho_{x_0,T},
\end{align*}
where we used the lower bound in \eqref{equation:eta.bound} to deduce the second inequality. Then, $e^{-Ct}\int_{\Gamma_t}(2-\eta)\rho_{x_0,T}\in[0,\infty)$ is non-increasing in $t\in(0,T)$, and its derivative is less than or equal to $-e^{-Ct}(C_T)^{-1}\int_{\Gamma_t}|A|^2\rho_{x_0,T}$. In particular,
\begin{equation}
\label{equation:A.wL2}
\int_{t
=T-\delta}^T\int_{\Gamma_t}|A|^2\rho_{x_0,T}\to0
\end{equation}
as $\delta\to0^+$. This control is key to our analysis.

To make use of \eqref{equation:A.wL2}, let $\lambda_k\to\infty$ be a sequence of positive numbers with $4\lambda_k^{-2}<T$. Let $\Gamma_s^k$ be the parabolic rescaling at $(x_0,T)$ in $\mathbb R^N\times\mathbb R$ of scale $\lambda_k$, with spacetime coordinates $y=\lambda_k(x-x_0)$ and $s=\lambda_k^2(t-T)$. Notice that we have scale invariance
\[
\int_{s
=-2}^{-1}\int_{\Gamma_s^k}|A|^2\rho_{0,0}
=\int_{\tau
=T-2\lambda_k^{-2}}^{T-\lambda_k^{-2}}\int_{\Gamma_\tau}|A|^2\rho_{x_0,T}.
\]
Thus, \eqref{equation:A.wL2} implies that
\[
\int_{s
=-2}^{-1}\int_{\Gamma_s^k}|A|^2\rho_{0,0}\to0\quad\text{and hence}\quad\int_{\Gamma_{s_k}^k}|A|^2\rho_{0,0}\to0
\]
as $k\to\infty$, where $s_k\in[-2,-1]$ is some sequence. Since $\rho_{0,0}$ has a uniform lower bound on $B_r(0)\times[-2,-1]$, we obtain:

\begin{lemma}
There exists a sequence  $s_k\in[-2,-1]$ such that
\begin{equation}
\label{equation:A.L2}
\int_{\Gamma_{s_k}^k\cap B_r(0)}|A|^2\to0
\end{equation}
as $k\to\infty$, for any $r>0$. By passing to a subsequence, we may assume $s_k\to s_\infty$ for some $s_\infty\in[-2,-1]$.
\end{lemma}

\subsection{Proof of Proposition~\ref{proposition:longtime}}

Assume $T<\infty$. Our goal is to show that $\Theta_{x_0,T}\le1$; from this we deduce that $(x_0,T)$ is a regular point by Section~4 in \cite{Whi05} and complete the proof.

Let $d_k$ denote the Euclidean distance between the origin and $\Gamma^k_{s_k}$.

\begin{lemma}
\label{lem:easy.case}
Suppose, after passing to a subsequence, that $d_k\to\infty$. Then, $\Theta_{x_0,T}=0<1$.
\end{lemma}

\begin{proof}
For any $r>0$, $|y|\ge r$, and $s\in[-2,-1]$, we observe that
\[
\rho_{0,0}(y,s)
=\frac1{4\pi|s|}e^{-\frac{|y|^2}{4|s|}}
\le\frac1{4\pi|s|}e^{-\frac{r^2}{16}}e^{-\frac{|y|^2}{8|s|}}
=2e^{-\frac{r^2}{16}}\rho_{0,|s|}(y,s);
\]
hence,
\begin{equation}
\label{equation:density.bound}
\int_{\Gamma_{s_k}^k\setminus B_r(0)}\rho_{0,0}
\le2e^{\frac{-r^2}{16}}\int_{\Gamma_{s_k}^k\setminus B_r(0)}\rho_{0,|s_k|}
\le2e^{\frac{-r^2}{16}}\int_{\Gamma_{s_k}^k}\rho_{0,|s_k|}
\le Ce^{\frac{-r^2}{16}},
\end{equation}
where in the last inequality we used Lemma~\ref{lem:gaussian.area}. Taking $r=d_k\to\infty$ and noting that $\int_{\Gamma_{s_k}^k}\rho_{0,0}=\int_{\Gamma_{s_k}^k\setminus B_{d_k}(0)}\rho_{0,0}$ by the definition of distance, we then obtain $\Theta_{x_0,T}=\lim_{k\to\infty}\int_{\Gamma_{s_k}^k}\rho_{0,0}=0$.
\end{proof}

By Lemma~\ref{lem:easy.case}, we may assume $d_k\le C$ for a uniform constant $C$ throughout the remainder of the section. In particular, $(x_0,T)$ is a limit point of $\iota(\cup_t\Gamma_t\times\{t\})$, and we have $x_0=\iota((p,\tilde p),T)$ for some $p\in\Sigma$ and $\tilde p\in\tilde\Sigma$.

Recall from Lemma~\ref{lem:eta}\eqref{lem:eta..} that $\Gamma_t$ is the graph of an area-preserving map $f_t:(\Sigma,g_t)\to(\tilde\Sigma,\tilde g_t)$ with operator norm $|df_t|_{\mathrm{op}}\le C_T$ on $\Sigma\times[0,T)$ for some $C_T>0$. In particular, $f_t$ is uniformly Lipschitz for $t\in[0,T)$. The condition $d_k\le C$ implies that there exists a sequence $p_k\in\Sigma$ satisfying
\begin{equation}
\label{equation:dist}
\operatorname{dist}_{G_{t_k}}((p_k,f_{t_k}(p_k)),(p,\tilde p))
\le C\lambda_k^{-1},
\end{equation}
where $t_k=T+\lambda_k^{-2}s_k$.

Take $U$ and $\tilde U$ to be open geodesic balls about $p$ and $\tilde p$ in $(\Sigma,g_T)$ and $(\tilde\Sigma,\tilde g_T)$. We may assume $f_{t_k}(U)\subset\tilde U$ since $p_k\to p$, $f_{t_k}(p_k)\to\tilde p$ and $f_{t_k}$ has a uniform Lipschitz constant. Let $\phi:\mathcal B_{r_U}(0)\to U$ and $\tilde\phi:\mathcal B_{r_{\tilde U}}(0)\to\tilde U$ be the geodesic charts, where $\mathcal B_r(0)$ denotes an open ball in $\mathbb R^2$ of radius $r$. We may also assume that the Euclidean metric is uniformly equivalent to $\phi^*g_{t_k}$ and $\tilde\phi^*\tilde g_{t_k}$ on $\mathcal B_{r_U}(0)$ and $\mathcal B_{r_{\tilde U}}(0)$. For clarity, we use $B$ for balls in $\mathbb R^N$, $\mathcal B$ for balls in $\mathbb R^2$, and $D$ for classical derivatives on $\mathcal B$. Also, we use ${|\cdot|}_{\mathrm{Euc}}$ for norms computed using the Euclidean metric and ${|\cdot|}$ without a subscript for norms associated with $G_t$; these two norms will be uniformly equivalent in the region under consideration. When integrating over $\mathcal B$, we always use the Euclidean area form.

Set $\phi_k=\phi(\lambda_k^{-1}\,\cdot\,)$ on $\mathcal B_{\lambda_kr_U}(0)$ and $\tilde\phi_k=\tilde\phi(\lambda_k^{-1}\,\cdot\,)$ on $\mathcal B_{\lambda_kr_{\tilde U}}(0)$. Let $g^k=\lambda_k^2\phi_k^*g_{t_k}$ and $\tilde g^k=\lambda_k^2\tilde\phi_k^*\tilde g_{t_k}$. Then $\phi_k$ and $\tilde\phi_k$ are isometries onto $(U,\lambda_k^2g_{t_k})$ and $(\tilde U,\lambda_k^2\tilde g_{t_k})$.

Define $u_k=\tilde\phi_k^{-1}\circ f_{t_k}\circ\phi_k:\mathcal B_{\lambda_kr_U}(0)\to\mathcal B_{\lambda_kr_{\tilde U}}(0)$. Next, define
\[
\iota_k:\mathcal B_{\lambda_kr_U}(0)\to\mathbb R^N,\quad z\mapsto\lambda_k\left(\iota((\phi_k(z),\tilde\phi_k(u_k(z))),t_k)-x_0\right),
\]
which parametrizes an open subset of $\Gamma^k_{s_k}$. Note that $\iota_k^*g_{\mathrm{Euc}}=(\mathrm{id}\times u_k)^*(g^k\oplus\tilde g^k)$ on $\mathcal B_{\lambda_kr_U}(0)$.

\begin{lemma}
\label{lem:uk}
We have $|Du_k|_{\mathrm{Euc}}\le C$ and $|u_k(0)|_{\mathrm{Euc}}\le C$. After passing to a subsequence, we have $u_k\to u_\infty$ in $C^\alpha_{\mathrm{loc}}(\mathbb R^2)$ for some function $u_\infty$.
\end{lemma}

\begin{proof}
The first two estimates come from $|du_k|_{\mathrm{op}}=|df_{t_k}|_{\mathrm{op}}\le C$, \eqref{equation:dist}, and the uniform Lipschitz bound for $f_{t_k}$. Then, the subsequential convergence is a direct consequence of the Arzel\`a--Ascoli theorem.
\end{proof}

\begin{lemma}
\label{lem:D2uk}
For every $r>0$,
\begin{equation}
\label{equation:D2u.L2}
\int_{\mathcal B_r(0)}|D^2u_k|_{\mathrm{Euc}}^2\,d\mathrm{area_{Euc}}\to0
\end{equation}
as $k\to\infty$.
\end{lemma}

\begin{proof}
Since $|u_k(0)|_{\mathrm{Euc}}\le C$ and $|Du_k|_{\mathrm{Euc}}\le C$, we have $|u_k|_{\mathrm{Euc}}\le C(1+r)$ on $\mathcal B_r(0)$. Since $\phi$ and $\tilde\phi$ are geodesic charts for $g_T$ and $\tilde g_T$, the Christoffel symbols of $g^k$ and $\tilde g^k$ are bounded by $C_r\lambda_k^{-2}$ on $\mathcal B_{C(1+r)}(0)$ for all sufficiently large $k$ for every $r>0$. This, together with the bound $|Du_k|_{\mathrm{Euc}}\le C$, implies
\[
|D^2u_k-\nabla du_k|_{\mathrm{Euc}}
\le C_r\lambda_k^{-2}
\]
on $\mathcal B_r(0)$, where $\nabla du_k$ denotes the Hessian of $u_k$ regarded as a map from $(\mathcal B_{\lambda_kr_U}(0),g^k)$ to $(\mathcal B_{\lambda_kr_{\tilde U}}(0),\tilde g^k)$. Consequently,
\[
|D^2u_k|_{\mathrm{Euc}}^2
\le(|\nabla du_k|_{\mathrm{Euc}}+C_r\lambda_k^{-2})^2
\le2(|\nabla du_k|_{\mathrm{Euc}}^2+C_r\lambda_k^{-4})
\]
and
\begin{equation*}
\int_{\mathcal B_r(0)}|D^2u_k|_{\mathrm{Euc}}^2\,d\mathrm{area_{Euc}}
\le2\int_{\mathcal B_r(0)}|\nabla du_k|_{\mathrm{Euc}}^2\,d\mathrm{area_{Euc}}+C_r\lambda_k^{-4}.
\end{equation*}
To prove \eqref{equation:D2u.L2}, it remains to show the first term on the right decays to zero.

Indeed, we have a general inequality
\begin{equation}
\label{equation:D2u.vs.A}
|A|
\le|\nabla du_k|
\le(1+|du_k|_{\mathrm{op}}^2)^{3/2}|A|
\end{equation}
that holds pointwise along the graph of $u_k$, where $|A|$ denotes the norm of the second fundamental form of $\Gamma^k_{s_k}$ in $\lambda_k(\iota(M\times\{t_k\})-x_0)$. Then,
\[
\int_{\mathcal B_r(0)}|\nabla du_k|_{\mathrm{Euc}}^2\,d\mathrm{area_{Euc}}
\le C\int_{\iota_k(\mathcal B_r(0))}|A|^2
\]
since $|du_k|_{\mathrm{op}}\le C$, $g^k$ and $\tilde g^k$ are uniformly equivalent to the Euclidean metric, and the area element of the graph of $u_k$ dominates that of $g^k$. Note that $\iota_k(\mathcal B_r(0))\subset\Gamma_{s_k}^k\cap B_{C(1+r)}(0)$. By \eqref{equation:A.L2}, the right-hand side converges to zero. This completes the proof of Lemma~\ref{lem:D2uk}.
\end{proof}

Define $v_{j,k}=\partial_{z_j}u_k$, so $|v_{j,k}|_{\mathrm{Euc}}\le C$ by Lemma~\ref{lem:uk}. By Poincar\'e's inequality, for every $r>0$,
\begin{equation}
\label{equation:Poincare}
\|v_{j,k}-a_{j,k,r}\|_{L^2(\mathcal B_r(0))}
\le C_r\|Dv_{j,k}\|_{L^2(\mathcal B_r(0))}\to0
\end{equation}
as $k\to\infty$, where the norms are computed with respect to the Euclidean metrics and areas, and
\[
a_{j,k,r}
=\frac1{\mathrm{vol}(\mathcal B_r(0))}\int_{\mathcal B_r(0)}v_{j,k}\,d\mathrm{area_{Euc}},\quad|a_{j,k,r}|
\le C.
\]
By \eqref{equation:Poincare}, after passing to a subsequence, we obtain $a_{j,k,r}\to a_j$ for some constant vectors $a_j$ independent of $r$ as $k\to\infty$, for every $r>0$.

The convergence $Du_k=(v_{1,k},v_{2,k})\to(a_1,a_2)$ in $L^2_{\mathrm{loc}}$ and the convergence $u_k\to u_\infty$ in $C^\alpha_{\mathrm{loc}}$ together imply that $u_\infty$ is an affine function. It remains to bound $\Theta_{x_0,T}$.

Let $L_1,L_2:\mathbb R^2\to\mathbb R^N$ denote the linear maps
\[
L_1
=d\iota|_{((p,\tilde p),T)}\circ d\phi|_0,\quad L_2
=d\iota|_{((p,\tilde p),T)}\circ d\tilde\phi|_0.
\]
Since $\iota$ is an isometry and the metric of $\mathcal M$ splits orthogonally, we have $L_1^TL_1=L_2^TL_2=I_{2\times2}$ and $L_1^TL_2=0$. Since the time variable is rescaled by $\lambda_k^{-2}$, Taylor's theorem and $|Du_k|_{\mathrm{Euc}}\le C$ give
\begin{equation}
\label{equation:iota}
\iota_k
=L_1+L_2u_k+o_{k\to\infty}(1),\quad D\iota_k
=L_1+L_2Du_k+o_{k\to\infty}(1),
\end{equation}
where the $o_{k\to\infty}(1)$ term is uniform on $\mathcal B_r(0)$ for every $r>0$. Combining \eqref{equation:iota} with Lemma~\ref{lem:uk} and Lemma~\ref{lem:D2uk}, we see that $\iota_k$ converges to an injection $\iota_\infty:=L_1+L_2u_\infty$ uniformly on $\mathcal B_r(0)$, and that the Jacobian $\mathcal J_k=\sqrt{\det((D\iota_k)^T(D\iota_k))}$ of $\iota_k$ converges to $\mathcal J_\infty=\sqrt{\det(I_{2\times2}+a^Ta)}$ in $L^1(\mathcal B_r(0))$ since the map $\xi\mapsto\sqrt{\det(I_{2\times2}+\xi^T\xi)}$ is Lipschitz on bounded sets of $M_{N\times2}(\mathbb R)$. Also, $\Gamma^k_{s_k}\cap B_r(0)\subset\iota_k(\mathcal B_{Cr}(0))$ for all $k$ sufficiently large. By the area formula and the uniform continuity of $\rho_{0,0}$ on $\mathbb R^N\times[-2,-1]$,
\begin{align*}
\limsup_{k\to\infty}\int_{\Gamma^k_{s_k}\cap B_r(0)}\rho_{0,0}
&\le\lim_{k\to\infty}\int_{\mathcal B_{Cr}(0)}\rho_{0,0}(\iota_k,s_k)\mathcal J_k\,d\mathrm{area_{Euc}}\\
&=\int_{\mathcal B_{Cr}(0)}\rho_{0,0}(\iota_\infty,s_\infty)\mathcal J_\infty \,d\mathrm{area_{Euc}},
\end{align*}
and the last integral is bounded by
\[
\int_{\iota_\infty(\mathbb R^2)}\rho_{0,0}(\,\cdot\,,s_\infty)
=e^{-\frac{\operatorname{dist}(0,\iota_\infty(\mathbb R^2))^2}{4|s_\infty|}}
\le1.
\]
With the tail estimate $\int_{\Gamma^k_{s_k}\setminus B_r(0)}\rho_{0,0}\le Ce^{-r^2/16}$ as in \eqref{equation:density.bound}, we obtain $\Theta_{x_0,T}\le1$ by letting $r\to\infty$. This completes the proof of Proposition~\ref{proposition:longtime}.

\section{Integral estimate for $|H|^2/\eta$}
\label{section:integral}

Following \cite{Wan08}, we derive a useful integral estimate. Let
\begin{equation}
\label{equation:I}
I(t)
=\int_{\Gamma_t}\frac{|H|^2}\eta
\ge\int_{\Gamma_t}|H|^2,
\end{equation}
with the second inequality following from $0<\eta\le1$. We compute
\begin{align*}
(\partial_t-\Delta)\frac{|H|^2}\eta
&=\frac{\eta(\partial_t-\Delta)|H|^2-|H|^2(\partial_t-\Delta)\eta}{\eta^2}\\
&+2\frac{\eta\nabla|H|^2\cdot\nabla\eta-|H|^2|\nabla\eta|^2}{\eta^3}.
\end{align*}
Plugging in \eqref{equation:eta.evolution} and \eqref{equation:H.evolution3} and rearranging terms, we obtain
\begin{align*}
(\partial_t-\Delta)\frac{|H|^2}\eta
&=\frac{(\kappa+(\kappa-\frac{R_M}4)(1-\eta^2))|H|^2+|H|^4+\mathcal E_1}\eta\\
&+2\frac{|H^\alpha h_{ij}^\alpha|^2-|H|^2|A|^2}\eta\\
&+2\frac{\eta\nabla|H|^2\cdot\nabla\eta-|H|^2|\nabla\eta|^2-\eta^2|\nabla H|^2}{\eta^3}.
\end{align*}
We now show that the last two terms are non-positive. Indeed, by the Cauchy--Schwarz inequality, $|H^\alpha h_{ij}^\alpha|^2-|H|^2|A|^2\le0$. Meanwhile,
\[
\eta\nabla|H|^2\cdot\nabla\eta-|H|^2|\nabla\eta|^2-\eta^2|\nabla|H||^2
=-\bigl||H|\nabla\eta-\eta\nabla|H|\bigr|^2,
\]
and $|\nabla H|\ge|\nabla|H||$. Thus,
\[
\eta\nabla|H|^2\cdot\nabla\eta-|H|^2|\nabla\eta|^2-\eta^2|\nabla H|^2
\le0.
\]
Therefore,
\[
(\partial_t-\Delta)\frac{|H|^2}\eta
\le\frac{(\kappa+(\kappa-\frac{R_M}4)(1-\eta^2))|H|^2+|H|^4+\mathcal E_1}\eta.
\]
Integrating this inequality and recalling \eqref{equation:area.evolution}, we find
\begin{equation}
\label{equation:I.evolution}
\frac d{dt}I(t)
\le\kappa I(t)+\int_{\Gamma_t}\frac{\mathcal E_2}\eta,
\end{equation}
where $\mathcal E_2=\mathcal E_1+((\kappa-\frac{R_M}4)(1-\eta^2)+(-\frac12G^{ij}\operatorname{Ric}_{ij}+\kappa))|H|^2$ and
\begin{equation}
\label{equation:E2.bound}
|\mathcal E_2|
\le Ce^{-\mu t}|H|(1+|A|).
\end{equation}
by \eqref{equation:g.convergence} and \eqref{ineq:E1}. Inequality \eqref{equation:I.evolution} is especially useful when $\kappa\le0$. Indeed, when $\kappa=0$, we know that $\eta\ge C^{-1}>0$ and so
\[
\frac d{dt}I(t)
\le Ce^{-\mu t}\int_{\Gamma_t}(1+|A|^2)
\le c_2e^{-\mu t}(I(t)+1),
\]
for some $c_2>0$, with the last inequality following from \eqref{ineq:A2}, \eqref{equation:I}, and the area bound \eqref{equation:area.bound}. This implies that
\begin{equation}
\label{equation:log.I}
\log(I(t)+1)+\frac{c_2}\mu e^{-\mu t}
\end{equation}
is non-increasing, in the case $\kappa=0$.

In the remainder of the section, we assume $\kappa<0$. First, recall from \cite{Ham88} that one may choose $\mu=|\kappa|=-\kappa$ in this case; note that this $\mu$ is half of the exponential decay rate in \cite{Ham88}, as our equation \eqref{equation:RF} differs from the one there. Then, by \eqref{equation:E2.bound} and the Cauchy--Schwarz inequality,
\[
\int_{\Gamma_t}\frac{\mathcal E_2}\eta
\le Ce^{\kappa t}I(t)^{\frac12}\left(\int_{\Gamma_t}\frac{1+|A|^2}\eta\right)^{\frac12}.
\]
By the lower bound for $\eta$ in Lemma~\ref{lem:eta}\eqref{lem:eta.<0} and the $L^2$ bound \eqref{ineq:A2} for $|A|$,
\begin{align*}
\int_{\Gamma_t}\frac{\mathcal E_2}\eta
&\le Ce^{\kappa t}\left(\inf_{\Gamma_t}\eta\right)^{-\frac12}I(t)^{\frac12}\left(\int_{\Gamma_t}(1+|A|^2)\right)^{\frac12}\\
&\le Ce^{\frac{\kappa t}2}I(t)^{\frac12}\left(\int_{\Gamma_t}(1+|H|^2)\right)^{\frac12}.
\end{align*}
Using the area bound \eqref{equation:area.bound} and the inequality $I(t)\ge\int|H|^2$ in \eqref{equation:I}, we find
\[
\int_{\Gamma_t}\frac{\mathcal E_2}\eta
\le Ce^{\frac{\kappa t}2}(I(t)^{\frac12}+I(t)).
\]
Combining this with \eqref{equation:I.evolution}, we obtain
\begin{equation}
\label{equation:I.evolution;kappa<0}
\frac d{dt}I(t)
\le\kappa I(t)+Ce^{\frac{\kappa t}2}(I(t)^{\frac12}+I(t)).
\end{equation}
Notice, in particular, that $I(t)$ is uniformly bounded for $t\in[0,\infty)$ because
\[
\kappa I(t)+Ce^{\frac{\kappa t}2}(I(t)^{\frac12}+I(t))
\le(\kappa+Ce^{\frac{\kappa t}2})I(t)<0
\]
whenever $I(t)\ge1$ and $t>0$ is sufficiently large. Therefore, \eqref{equation:I.evolution;kappa<0} gives
\[
\frac d{dt}(e^{-\kappa t}I(t))
\le C(e^{-\kappa t}I(t))^{\frac12}(1+I(t)^{\frac12})
\le C(e^{-\kappa t}I(t))^{\frac12}.
\]
We conclude that, in the case $\kappa<0$,
\begin{equation}
\label{equation:I.bound;kappa<0}
I(t)
\le C(1+t^2)e^{\kappa t}.
\end{equation}

\section{Uniform bounds on the second fundamental form}
\label{section:A.bound}

Proposition~\ref{proposition:longtime} says that the mean curvature flow exists for all time, so the second fundamental form is uniformly bounded on every finite time interval. In this section, we upgrade this to a bound uniform in time:

\begin{proposition}
\label{proposition:A.bound}
We have $\sup_{t\ge0}\sup_{\Gamma_t}|\nabla^kA|\le C_k$ for some $C_k>0$ for each $k\ge0$. Moreover, $\sup_{\Gamma_t}|H|\to0$ as $t\to\infty$. In fact, $\sup_{\Gamma_t}|A|\to0$ as $t\to\infty$ if $\kappa\ge0$.
\end{proposition}

By the uniform bounds on $|\nabla^kA|$ in the proposition and on $\operatorname{area}(\Gamma_t)$ in \eqref{equation:area.bound}, every sequence $\Gamma_{t_k}$ with $t_k\to\infty$ admits a smoothly convergent subsequence. The limit of the subsequence is minimal because $\sup_{\Gamma_t}|H|\to0$, and is Lagrangian because $\omega_t'|_{\Gamma_t}=0$ and $\omega_t'\to\omega_\infty'$ smoothly. We thus obtain:

\begin{corollary}
\label{corollary:subsequence}
There exist a smooth minimal Lagrangian surface $\Gamma_\infty$ in $(M,G_\infty,\omega'_\infty)$ and a sequence of times $t_k\to\infty$ such that $\Gamma_{t_k}\to\Gamma_\infty$ smoothly.
\end{corollary}

The arguments of \cites{Wan01a, Wan01b, Wan08} in the Ricci-flat background case essentially carry over to prove Proposition~\ref{proposition:A.bound}, but we include the details for the reader's convenience. In the case $\kappa<0$, we modify the proof by replacing the use of \cite{Ni02} with a classical result on the parabolicity of surfaces with quadratic area growth.

Once we prove $|A|\le C$ uniformly in time, a standard bootstrapping argument in mean curvature flow then establishes bounds for higher-order derivatives $|\nabla^kA|\le C_k$. Thus, the main goal in the remainder of the section is to obtain the desired bounds on $|A|$ and $|H|$.

\subsection{Case $\kappa>0$}

In this case, $\eta\to1$ uniformly as $t\to\infty$ by Lemma~\ref{lem:eta}\eqref{lem:eta.>0}. Thus, there exists some $T_0>0$ such that for any $t\ge T_0$, we have $R_M>0$, $\eta>\frac12$, and $8\sqrt{2(1-\eta^2)}\le\min(\frac1{c_1},\frac16)$. Taking $p=c_1$ in \eqref{equation:etaA.evolution} gives
\begin{align*}
\eta^{2p}(\partial_t-\Delta)(\eta^{-2p}|A|^2)
&\le c_1|A|-\frac16c_1|A|^4
\end{align*}
for $t\ge T_0$. Since $\eta>\frac12$ for $t\ge T_0$, the maximum principle gives a uniform bound for $|A|$ on $\Sigma\times[T_0,\infty)$ and hence on $\Sigma\times[0,\infty)$.

Then, any sequence $\Gamma_{t_k}$ with $t_k\to\infty$ has a smoothly convergent subsequence. The limit $\Gamma_\infty$ of such a subsequence must be the graph of an isometry because $\eta\to1$ as $t\to\infty$ by Lemma~\ref{lem:eta}\eqref{lem:eta.>0}; in particular, $\Gamma_\infty$ is totally geodesic. This ensures that $\sup_{\Gamma_t}|A|\to0$ as $t\to\infty$, for otherwise we can find a subsequence whose limit is not totally geodesic.

\subsection{Case $\kappa=0$}

In this case, $\eta$ has a uniform lower bound $C^{-1}$ independent of time by Lemma~\ref{lem:eta}\eqref{lem:eta.=0}. Therefore,
\[
\int_{\Gamma_t}|H|^2
\le I(t)
=\int_{\Gamma_t}\frac{|H|^2}\eta
\le C\int_{\Gamma_t}|H|^2.
\]
Meanwhile, we recall from \eqref{equation:H.L2} that
\[
\int_0^\infty\int_{\Gamma_t}|H|^2<\infty.
\]
This estimate and the monotonicity of \eqref{equation:log.I} imply $I(t)\to0$ and $\int_{\Gamma_t}|H|^2\to0$ as $t\to\infty$. By the Gauss--Bonnet theorem \eqref{equation:GB} and the exponential convergence $G_t\to G_\infty$ in \eqref{equation:g.convergence}, $\int_{\Gamma_t}|A|^2\to0$ as $t\to\infty$. The small $\epsilon$-regularity theorem \cite{Ilm95}*{Theorem~14} then gives a time-independent bound on $|A|$, and hence on $|\nabla^kA|$. Smooth subsequential compactness of $\Gamma_t$ and $\int_{\Gamma_t}|A|^2\to0$ then imply that $\sup_{\Gamma_t}|A|\to0$ as $t\to\infty$.

\subsection{Case $\kappa<0$}

In this case, we apply a blow-up argument; with minor modifications, the reasoning actually applies when $\kappa\ge0$. First, recall from \eqref{equation:I.bound;kappa<0} that
\begin{equation}
\label{ineq:H;kappa<0}
\int_{\Gamma_t}|H|^2
\le\int_{\Gamma_t}\frac{|H|^2}\eta
\le C(1+t^2)e^{\kappa t},
\end{equation}
which implies $\int_{\Gamma_t}|H|^2\to0$. Below, we prove $\sup_{\Gamma_t}|A|\le C$. Once this is established, we obtain $\sup_{\Gamma_t}|\nabla^kA|\le C_k$, so smooth subsequential compactness and $\int_{\Gamma_t}|H|^2\to0$ imply that $\sup_{\Gamma_t}|H|\to0$ as $t\to\infty$.

Now, suppose to the contrary that there exists a sequence of spacetime points $(p_k,t_k)$ with $p_k\in\Gamma_{t_k}$, $t_k\to\infty$, and
\[
\lambda_k:
=|A(p_k,t_k)|
=\max_{t\in[0,t_k]}\max_{p\in\Gamma_t}|A(p,t)|\to\infty.
\]
By a standard blow-up procedure at $(p_k,t_k)$ with blow-up factor $\lambda_k$, we obtain a connected, complete, smooth, immersed, ancient mean curvature flow of surfaces $L_t$ passing through $(0,0)$ in $\mathbb R^4\times(-\infty,0]$ with $|A|\le1$ and $|A(0,0)|=1$. Since the integral $\int|H|^2$ is scale-invariant, the decay \eqref{ineq:H;kappa<0} ensures that each time slice $L_t$ is minimal, so $L_t=L_0$ for $t\le0$. Note also that $L_0$ is Lagrangian with respect to $\overline\omega_1:=dx_1\wedge dx_2-dy_1\wedge dy_2$, where $(x_1,x_2,y_1,y_2)$ are the Euclidean coordinates of $\mathbb R^4$.

Under this blow-up procedure, $\eta$ converges locally smoothly to a time-independent function $\overline\eta=*\left(2dx_1\wedge dx_2|_{L_0}\right)\in[0,1]$ on $L_0$, and the evolution equation \eqref{equation:eta.evolution} passes to the limit to give $\Delta\overline\eta=-2|A|^2\overline\eta$. In particular, $\overline\eta$ is a non-negative superharmonic function on $L_0$.

On the other hand, by the area growth bound in Corollary~\ref{corollary:area.growth}, the area growth of the blow-up limit $L_0$ is at most quadratic with respect to the extrinsic Euclidean distance; that is,
\[
\mathrm{area}(L_0\cap B_R(0))
\le CR^2
\]
for all $R>0$. Since intrinsic distances on $L_0$ are bounded below by extrinsic ones, the area growth of $L_0$ is also at most quadratic with respect to the intrinsic distance. It is a classical result (see, e.g., \cite{CY75}*{Corollary~1}) that every non-negative superharmonic function on a complete Riemannian manifold whose volume growth is at most quadratic must be constant. Hence, $\overline\eta$ is constant, and $|A|^2\overline\eta=-\frac12\Delta\overline\eta=0$. Since $|A(0,0)|=1$, $\overline\eta=0$ identically.

Since $\overline\eta=0$ and $\overline\omega_1|_{L_0}=0$, we have $dx_1\wedge dx_2|_{L_0}=dy_1\wedge dy_2|_{L_0}=0$. Let $\pi$ and $\tilde\pi$ be the projections sending $(x_1,x_2,y_1,y_2)$ to $(x_1,x_2)$ and $(y_1,y_2)$, respectively. Then, at every point on $L_0$, the linear maps $d\pi|_{L_0}$ and $d\tilde\pi|_{L_0}$ both have rank at most $1$. Since the sum of their ranks is at least $\dim L_0=2$,
\[
\operatorname{rank}(d\pi|_{L_0})
=\operatorname{rank}(d\tilde\pi|_{L_0})
=1.
\]
Now, fix a preimage $p\in L_0$ of $0$ under the immersion, and let $B_r=B_r^{L_0}(p)$ denote the intrinsic ball about $p$ of radius $r$ in $L_0$.
By the constant rank theorem, $\pi(B_r)$ and $\tilde\pi(B_r)$ are connected smooth curves in $\mathbb R^2$ if $r>0$ is sufficiently small. Also, $B_r\subset\pi(B_r)\times\tilde\pi(B_r)$. Since $L_0$ is minimal, both $\pi(B_r)$ and $\tilde\pi(B_r)$ must be geodesics in $\mathbb R^2$. Consequently, $B_r$ lies in a plane, contradicting $|A(0,0)|=1$. This completes the proof of Proposition~\ref{proposition:A.bound}.

\section{Long-time exponential convergence}
\label{section:convergence}

The convergence in Corollary~\ref{corollary:subsequence} is only subsequential, and it remains to show that $\Gamma_t\to\Gamma_\infty$ as $t\to\infty$. As in \cite{Wan01b}, one approach is to apply Simon's \L{}ojasiewicz-type inequality \cite{Sim83}. We instead rely on energy estimates. This allows us to prove not only the smooth convergence but also exponential convergence. Another possible approach can be found in Section~6 of Part~II of \cite{Sim85}.

\begin{proposition}
\label{proposition:convergence}
For any sufficiently small $\mu_0>0$, there exists a constant $C_{\mu_0}>0$ with
\[
\int_{\Gamma_t}|H|^2
\le C_{\mu_0}e^{-2\mu_0t}.
\]
\end{proposition}

Together with the uniform bounds in Proposition~\ref{proposition:A.bound} and interpolation inequalities, Proposition~\ref{proposition:convergence} has the following  corollary.

\begin{corollary}
\label{corollary:convergence}
There exists a smooth minimal Lagrangian surface $\Gamma_\infty$ in $(M,G_\infty,\omega'_\infty)$ such that $\Gamma_t$ converges smoothly and exponentially to $\Gamma_\infty$ as $t\to\infty$.
\end{corollary}

\begin{remark}
\label{remark:convergence}
Let $\nu_*(\Gamma_\infty)$ be the smallest positive eigenvalue of $-\Delta_{\mathrm H}-\kappa\mathrm{Id}$ on $\Omega^1(\Gamma_\infty)$. Then, in Proposition~\ref{proposition:convergence}, $\mu_0>0$ could be any number with $\mu_0\le\mu$ and $\mu_0<\nu_*(\Gamma_\infty)$.
\end{remark}

\begin{remark}
Estimate \eqref{equation:I.bound;kappa<0} already gives an exponential decay for $\kappa<0$, but the rate there is not optimal; this section gives a better decay rate.
\end{remark}

We now begin the proof of Proposition~\ref{proposition:convergence}. The main idea is as follows. We compute the evolution equation of $\int_{\Gamma_t}|H|^2$. Under the identification $X\mapsto\alpha_X$ of normal fields with $1$-forms in Subsection~\ref{subsection:normal}, the leading term of this equation is governed by the operator $-\Delta_{\mathrm H}-\kappa\mathrm{Id}$ on $\Omega^1(\Gamma_t)$ (which, on $\Gamma_\infty$, indeed coincides with the Jacobi operator \cite{Oh90}*{Proposition~4.1}). Since this operator is asymptotically non-negative, the exponential decay of $\int_{\Gamma_t}|H|^2$ would follow from a spectral gap, provided that the component of $\alpha_H$ in the low eigenspaces is sufficiently small. When $\kappa<0$, the gap is straightforward. When $\kappa\ge0$, the relevant low eigenspaces require additional consideration. The closed low eigenforms are, up to small errors, induced by Killing fields of the limiting ambient metric; thus, $\alpha_H$ is almost $L^2$-orthogonal to these forms because the first variation of area along such Killing fields nearly vanishes. When $\kappa>0$, there are also co-exact low eigenforms; these are almost $L^2$-orthogonal to $\alpha_H$ because $d\alpha_H$ is small by the Lagrangian condition.

\subsection{Normal vector fields as $1$-forms}
\label{subsection:normal}

The Lagrangian condition on $\Gamma_t$ ensures that the complex structure $J$ on $M$ restricts to an isomorphism from $T\Gamma_t$ to $N\Gamma_t$. This allows us to identify sections of the normal bundle $N\Gamma_t$ with differential $1$-forms on $\Gamma_t$. Indeed, given any section $X$ of $N\Gamma_t$, we can set $\alpha_X=\omega'_t(X,\,\cdot\,)|_{\Gamma_t}=G_t(JX,\,\cdot\,)\in\Omega^1(\Gamma_t)$, where $JX$ is a section of $T\Gamma_t$. If $X$ and $Y$ are two such sections, then by definition we have $\langle\alpha_X,\alpha_Y\rangle=\langle JX,JY\rangle$, which equals $\langle X,Y\rangle$ since $J$ is an isometry. Since $\overline\nabla J=0$, we also have $\langle\nabla\alpha_X,\nabla\alpha_Y\rangle=\langle\nabla X,\nabla Y\rangle$, where $\nabla\alpha_X$ is the intrinsic covariant derivative of $\alpha_X$ on $\Gamma_t$, and $\nabla X$ is the covariant derivative of $X$ with respect to the normal connection. In particular, $|\nabla\alpha_H|=|\nabla H|$.

More generally, for any (not necessarily normal) vector field $X$ along $\Gamma_t$ in $M$, we set $\alpha_X=\omega'_t(X,\,\cdot\,)|_{\Gamma_t}$, which is equal to $\alpha_{X^\perp}$ by the Lagrangian condition. Then, $\langle\alpha_X,\alpha_Y\rangle=\langle X^\perp,Y^\perp\rangle$, and if either $X$ or $Y$ is a section of the normal bundle, we again have $\langle\alpha_X,\alpha_Y\rangle=\langle X,Y\rangle$.

Let $\Delta_{\mathrm H}=-(dd^*+d^*d)$ denote the Hodge Laplacian. Recall that $K=\frac18R_M\eta^2+\frac12|H|^2-\frac12|A|^2$ as in \eqref{equation:K}. The Weitzenb\"ock formula says
\begin{equation}
\label{equation:Weitzenbock}
\int_{\Gamma_t}|\nabla\alpha_H|^2
=\int_{\Gamma_t}\langle-\Delta_{\mathrm H}\alpha_H,\alpha_H\rangle-\int_{\Gamma_t}K|\alpha_H|^2.
\end{equation}

\subsection{Differential equation for $\|H\|_{L^2}^2$}

By the evolution equations \eqref{equation:area.evolution} and \eqref{equation:H.evolution3} for the area form and $|H|^2$, we have
\begin{equation*}
\frac d{dt}\int_{\Gamma_t}|H|^2
=\int_{\Gamma_t}(-2|\nabla H|^2+2|H^\alpha h_{ij}^\alpha|^2+\kappa(2-\eta^2)|H|^2-|H|^4+\mathcal E),
\end{equation*}
where
\begin{equation}
\label{ineq:E}
|\mathcal E|
\le Ce^{-\mu t}|H|(1+|A|)
\le Ce^{-\mu t}|H|.
\end{equation}
Note that the second inequality follows from the uniform bound $|A|\le C$ in Proposition~\ref{proposition:A.bound}. Using \eqref{equation:Weitzenbock}, we obtain
\begin{align*}
\frac d{dt}\int_{\Gamma_t}|H|^2
&=2\int_{\Gamma_t}\langle\Delta_{\mathrm H}\alpha_H,\alpha_H\rangle+\int_{\Gamma_t}(2|H^\alpha h_{ij}^\alpha|^2-|A|^2|H|^2)\\
&+\int_{\Gamma_t}(\kappa(2-\eta^2)+\frac14R_M\eta^2)|H|^2+\int_{\Gamma_t}\mathcal E.
\end{align*}
Note that $\kappa(2-\eta^2)+\frac14R_M\eta^2\to2\kappa$ since $\frac14R_M\to\kappa$, and that
\[
\int_{\Gamma_t}\mathcal E
\le Ce^{-\mu t}\Bigl(\int_{\Gamma_t}|H|^2\Bigr)^{\frac12}
\]
by \eqref{ineq:E}, the Cauchy--Schwarz inequality, and the area bound \eqref{equation:area.bound}. Also, recall from Lemma~\ref{lem:A.symmetry}\eqref{lem:A.symmetry2} that
\[
\bigl|2|H^\alpha h_{ij}^\alpha|^2-|A|^2|H|^2\bigr|
\le\frac2{\sqrt3}|A|\,|H|^3;
\]
the right-hand side is $o(1)|H|^2$ since $|A|\le C$ and $|H|\to0$ by Proposition~\ref{proposition:A.bound}, where $o(1)=o_{t\to\infty}(1)$ denotes a quantity tending to zero as $t\to\infty$.

In summary,
\begin{align}
\label{equation:H.L2.evolution}
\frac d{dt}\|H\|_{L^2(\gamma_t)}^2
&\le2\langle(\Delta_{\mathrm H}+\kappa)\alpha_H,\alpha_H\rangle_{L^2(\gamma_t)}\\
\nonumber&+o(1)\|H\|_{L^2(\gamma_t)}^2+Ce^{-\mu t}\|H\|_{L^2(\gamma_t)},
\end{align}
where $\gamma_t=(\operatorname{Id},f_t)^*G_t=g_t+f_t^*\tilde g_t$ denotes the pullback metric on $\Sigma$.

\subsection{Projection onto the low eigenspaces}

Roughly speaking, in order to estimate the term $\langle(\Delta_{\mathrm H}+\kappa)\alpha_H,\alpha_H\rangle_{L^2(\gamma_t)}$ in \eqref{equation:H.L2.evolution}, we want to show that the $L^2$ projection of $\alpha_H$ onto the low eigenspace of $-\Delta_{\mathrm H}-\kappa$ is sufficiently small. To this end, we first show the following.

\begin{lemma}
\label{lem:small}
For all $t\ge0$,
\begin{enumerate}[(i)]
\item$|d\alpha_H|\le Ce^{-\mu t}$ on $\Gamma_t$;

\item for every Killing field $X$ of $(M,G_\infty)$, $\lvert\langle\alpha_H,\alpha_X\rangle_{L^2(\gamma_t)}\rvert\le Ce^{-\mu t}\|X\|_{C^1}$.
\end{enumerate}
\end{lemma}

\begin{proof}
(i) For a Lagrangian submanifold $\Gamma_t$ in a K\"ahler manifold $(M,G_t,\omega_t')$, Dazord \cite{Daz81}*{p.~477} proved that
\[
d\alpha_H
=\rho_M|_{\Gamma_t},
\]
where $\rho_M$ denotes the Ricci form of $(M,G_t,\omega_t')$ defined by $\rho_M(Y,Z)=\operatorname{Ric}(JY,Z)$ for tangent vectors $Y,Z$ on $M$. Then, $\rho_M=\frac12R_\Sigma\omega_t-\frac12R_{\tilde\Sigma}\tilde\omega_t$. By the Lagrangian condition $\omega_t|_{\Gamma_t}=\tilde\omega_t|_{\Gamma_t}$,
\[
d\alpha_H
=\frac12(R_\Sigma-R_{\tilde\Sigma})\omega_t|_{\Gamma_t}.
\]
Since $\lvert\omega_t|_{\Gamma_t}\rvert=\frac\eta2\le\tfrac12$, (i) follows from the convergence \eqref{equation:g.convergence} of $g_t$ and $\tilde g_t$.

(ii) Recall that $\int_{\Gamma_t}\langle\alpha_H,\alpha_X\rangle=\int_{\Gamma_t}\langle H,X\rangle$. By the divergence theorem, the latter equals
\[
-\int_{\Gamma_t}\mathrm{div}_{\Gamma_t}X
=-\frac12\int_{\Gamma_t}\operatorname{tr}_{T\Gamma_t}(\mathcal L_XG_t)
=\frac12\int_{\Gamma_t}\operatorname{tr}_{T\Gamma_t}\mathcal L_X(G_\infty-G_t),
\]
where the last equality uses the Killing condition $\mathcal L_XG_\infty=0$. Now, (ii) follows from \eqref{equation:g.convergence} and \eqref{equation:area.bound}. This completes the proof of Lemma~\ref{lem:small}.
\end{proof}

Now, define $\Lambda_t$ and $\bar\nu_t$ according to the sign of $\kappa$:

\begin{itemize}
\item If $\kappa>0$, let $\Lambda_t$ be the span of the eigen $1$-forms of $-\Delta_{\mathrm H}$ on $\Gamma_t$ with eigenvalue at most $2\kappa$, and let $\bar\nu_t$ be the smallest eigenvalue greater than $2\kappa$.

\item If $\kappa=0$, let $\Lambda_t$ be the space of harmonic $1$-forms, and let $\bar\nu_t$ be the smallest positive eigenvalue of $-\Delta_{\mathrm H}$.

\item If $\kappa<0$, set $\Lambda_t=\{0\}$ and $\bar\nu_t=0$.
\end{itemize}

Then, let $P_t$ be the $L^2(\gamma_t)$-orthogonal projection from $\Omega^1(\Gamma_t)$ onto $\Lambda_t$. With these definitions,
\begin{equation}
\label{equation:gap}
\langle(\Delta_{\mathrm H}+\kappa)\alpha_H,\alpha_H\rangle_{L^2(\gamma_t)}
\le-(\bar\nu_t-\kappa)\|H\|_{L^2(\gamma_t)}^2+\bar\nu_t\|P_t\alpha_H\|_{L^2(\gamma_t)}^2.
\end{equation}

\begin{lemma}
\label{lem:spectral}
There exists $C>0$ such that
\begin{enumerate}[(i)]
\item$\bar\nu_t\le C$,

\item$\bar\nu_\infty:=\liminf_{t\to\infty}\bar\nu_t>\kappa$,

\item$\|P_t\alpha_H\|_{L^2(\gamma_t)}^2\le Ce^{-2\mu t}+o(1)\|H\|_{L^2(\gamma_t)}^2$.
\end{enumerate}
\end{lemma}

\begin{proof}
We distinguish three cases according to the sign of $\kappa$.

\emph{Case 1: $\kappa>0$.} Since $\eta_t\to1$ as $t\to\infty$ by Lemma~\ref{lem:eta}\eqref{lem:eta.>0}, the singular values of $df_t$ tend to $1$. Combining this with \eqref{equation:g.convergence}, we find
\[
\|\gamma_t-2g_\infty\|_{C^0(\Sigma)}\to0;
\]
note that $2g_\infty$ is a round metric of scalar curvature $\kappa$. By the Rayleigh quotient formula, the eigenvalues of $-\Delta_{\gamma_t}$ on scalar functions converge to those of $-\Delta_{2g_\infty}$, namely $0,\kappa,3\kappa,\dots$ with multiplicities $1,3,5,\dots$.

On the other hand, since $H^1(\Gamma_t;\mathbb R)=0$, the Hodge decomposition shows that there are no harmonic $1$-forms on $\Gamma_t$; moreover, for each $\nu>0$, the $\nu$-eigenspace of $-\Delta_{\mathrm H}$ is spanned by the exact forms $d\varphi$ and the co-exact forms $*d\varphi$, where $\varphi$ ranges over the $\nu$-eigenspace of $-\Delta_{\gamma_t}$ on functions. Thus,
\[
\bar\nu_\infty
=\lim_{t\to\infty}\bar\nu_t
=3\kappa,
\]
which proves (i) and (ii).

Let $\Lambda_t^{\rm ex}$ and $\Lambda_t^{\rm co}$ be the subspaces of $\Lambda_t$ consisting of exact and co-exact forms, respectively. Then, for $t$ sufficiently large, $\Lambda_t=\Lambda_t^{\rm ex}\oplus\Lambda_t^{\rm co}$ with $\dim\Lambda_t^{\rm ex}=\dim\Lambda_t^{\rm co}=3$, all eigenvalues corresponding to $\Lambda_t$ lie in $[\frac\kappa2,2\kappa]$, and all others are larger than $2\kappa$. Next, to estimate $\|P_t\alpha_H\|_{L^2(\gamma_t)}$, we compute the inner product of $\alpha_H$ with exact and co-exact eigen $1$-forms.

First, let $\beta\in\Lambda_t^{\rm co}$ with $\|\beta\|_{L^2(\gamma_t)}=1$ and $\Delta_{\mathrm H}\beta=-\nu\beta$ for some $\nu\in[\frac\kappa2,2\kappa]$. Then, $d^*\beta=0$ and
\[
\nu\langle\alpha_H,\beta\rangle_{L^2(\gamma_t)}
=\langle\alpha_H,-\Delta_{\mathrm H}\beta\rangle_{L^2(\gamma_t)}
=\langle\alpha_H,d^*d\beta\rangle_{L^2(\gamma_t)}
=\langle d\alpha_H,d\beta\rangle_{L^2(\gamma_t)};
\]
similarly, $\nu=\nu\|\beta\|_{L^2(\gamma_t)}^2=\|d\beta\|_{L^2(\gamma_t)}^2$. Thus,
\begin{equation}
\label{equation:cc}
\lvert\langle\alpha_H,\beta\rangle_{L^2(\gamma_t)}\rvert
\le\frac1{\nu}\|d\alpha_H\|_{L^2(\gamma_t)}\|d\beta\|_{L^2(\gamma_t)}
\le Ce^{-\mu t}
\end{equation}
by the Cauchy--Schwarz inequality, Lemma~\ref{lem:small}(i), and the area bound \eqref{equation:area.bound}.

Second, let $\beta\in\Lambda_t^{\rm ex}$ with $\|\beta\|_{L^2(\gamma_t)}=1$ and $\Delta_{\mathrm H}\beta=-\nu\beta$ for some $\nu=\kappa+o(1)$. Then, $\beta=\nu^{-1/2}d\varphi$ with $-\Delta_{\gamma_t}\varphi=\nu\varphi$, $\|\varphi\|_{L^2(\gamma_t)}=1$, and $\int_{\Gamma_t}\varphi=0$. Expand $\varphi=b_0+\sum_jb_jh_j+w$ orthogonally in $L^2(2g_\infty)$, where $b_0,\dots,b_3$ are constants, $\{h_j\}_{j=1}^3$ is an orthonormal basis of the first eigenspace in $C^\infty(\Sigma)$ of $(\Sigma,2g_\infty)$, and $w$ lies in the span of the eigenfunctions with eigenvalues $\ge3\kappa$. Since $\gamma_t\to2g_\infty$ in $C^0$, the $L^2$ norms of functions and $1$-forms with respect to $\gamma_t$ and $2g_\infty$ agree up to a factor $1+o(1)$. In particular,
\[
\|\varphi\|_{L^2(2g_\infty)}^2
=1+o(1),\quad\|d\varphi\|_{L^2(2g_\infty)}^2
=\kappa+o(1).
\]
Since $\|dw\|_{L^2(2g_\infty)}^2\ge3\kappa\|w\|_{L^2(2g_\infty)}^2$, we necessarily have
\begin{equation}
\label{equation:high.mode}
b_0
=o(1),\quad\|w\|_{L^2(2g_\infty)}^2
=o(1),\quad\|dw\|_{L^2(2g_\infty)}^2
=o(1),
\end{equation}
and $|b_j|\le C$. Note that $X_j:=-J(\nabla^{g_\infty}h_j,0)$ is a Killing field of $(M,G_\infty)$; to see this, recall that the first eigenfunctions of $-\Delta$ on the unit sphere $S^2\subset\mathbb R^3$ are precisely the restrictions of nonzero linear functions in $\mathbb R^3$. A direct computation and the convergence \eqref{equation:g.convergence} show
\begin{equation}
\label{equation:alphaX}
\|d(h_j\circ\pi)-\alpha_{X_j}\|_{C^0(\gamma_t)}
\le Ce^{-\mu t},
\end{equation}
where $\pi$ denotes the projection from $\Gamma_t\subset\Sigma\times\tilde\Sigma$ onto the first component $\Sigma$. Therefore, by Lemma~\ref{lem:small}(ii), \eqref{equation:alphaX}, the area bound \eqref{equation:area.bound}, and \eqref{equation:high.mode},
\begin{align}
\label{equation:ex}
\lvert\langle\alpha_H,\beta\rangle_{L^2(\gamma_t)}\rvert
&\le C\sum_j\Bigl(\lvert\langle\alpha_H,\alpha_{X_j}\rangle_{L^2(\gamma_t)}\rvert\\
\nonumber&+\|\alpha_H\|_{L^2(\gamma_t)}\|d(h_j\circ\pi)-\alpha_{X_j}\|_{L^2(\gamma_t)}\Bigr)\\
\nonumber&+\|\alpha_H\|_{L^2(\gamma_t)}\|dw\|_{L^2(\gamma_t)}\\
\nonumber&\le Ce^{-\mu t}+o(1)\|H\|_{L^2(\gamma_t)}.
\end{align}
Summing the squares of \eqref{equation:cc} and \eqref{equation:ex} over an orthonormal basis of $\Lambda_t$, we obtain (iii).

\emph{Case 2: $\kappa=0$.} Here, $\Lambda_t=\ker(\Delta_{\mathrm H})$ is the $2$-dimensional space of harmonic $1$-forms, and $\bar\nu_t$ is the first positive eigenvalue of $-\Delta_{\mathrm H}$. By the Hodge decomposition, $\bar\nu_t$ is also the first positive eigenvalue of $-\Delta_{\gamma_t}$ on scalar functions. Since $\eta\ge C^{-1}$ by Lemma~\ref{lem:eta}\eqref{lem:eta.=0}, the metrics $\gamma_t$ are uniformly bi-Lipschitz to $g_\infty$. This ensures that (i) and (ii) hold.

We know that $|\nabla df_t|\to0$ uniformly by Proposition~\ref{proposition:A.bound} and \eqref{equation:D2u.vs.A}, and that the metrics $g_t,\tilde g_t$ converge as in \eqref{equation:g.convergence}. Therefore, in the flat coordinates of $(\Sigma,g_\infty)$, the coefficients of $\gamma_t=g_t+f_t^*\tilde g_t$ oscillate by at most $o(1)$; more precisely, if we fix $p_0\in\Sigma$, the constant-coefficient flat metric $\hat\gamma_t:=\gamma_t(p_0)$ on $\Sigma$ satisfies $\|\gamma_t-\hat\gamma_t\|_{C^0}=o(1)$. In particular, the $L^2$ norms of $1$-forms with respect to $\gamma_t$ and $\hat\gamma_t$ agree up to a factor $1+o(1)$.

Let $\beta\in\Lambda_t$ with $\|\beta\|_{L^2(\gamma_t)}=1$, and let $\hat\beta$ be the unique $\hat\gamma_t$-harmonic $1$-form in the class $[\beta]\in H^1(\Sigma,\mathbb R)$. Since the harmonic representative minimizes the $L^2$ norm within its cohomology class, we have $\|\beta\|_{L^2(\gamma_t)}^2\le\|\hat\beta\|_{L^2(\gamma_t)}^2$ and $\|\hat\beta\|_{L^2(\hat\gamma_t)}^2\le\|\beta\|_{L^2(\hat\gamma_t)}^2$. Since the $L^2$ norms with respect to the two metrics agree up to a factor $1+o(1)$, we have $\|\hat\beta\|_{L^2(\gamma_t)}^2=1+o(1)$. Then, since $\hat\beta-\beta$ is exact and $\beta$ is $\gamma_t$-harmonic,
\begin{equation}
\label{equation:harm.const}
\|\hat\beta-\beta\|_{L^2(\gamma_t)}^2
=\|\hat\beta\|_{L^2(\gamma_t)}^2-\|\beta\|_{L^2(\gamma_t)}^2
=o(1).
\end{equation}
Next, let $v$ be the constant vector field on $\Sigma$ with $\omega_\infty(v,\,\cdot\,)=\hat\beta(\,\cdot\,)$, so that $|v|\le C$. Then, let $X=(v,0)$ be a Killing field of $(M,G_\infty)$. Then,
\begin{equation}
\label{equation:alphaX.torus}
\|\alpha_X-\hat\beta\|_{C^0(\gamma_t)}
=\|\iota_v(\omega_t-\omega_\infty)\|_{C^0(\gamma_t)}
\le Ce^{-\mu t}
\end{equation}
because $\omega_t$ converges exponentially to $\omega_\infty$ by \eqref{equation:g.convergence}. Therefore,
\begin{align*}
\lvert\langle\alpha_H,\beta\rangle_{L^2(\gamma_t)}\rvert
&\le\lvert\langle\alpha_H,\alpha_X\rangle_{L^2(\gamma_t)}\rvert\\
&+\|\alpha_H\|_{L^2(\gamma_t)}(\|\alpha_X-\hat\beta\|_{L^2(\gamma_t)}+\|\hat\beta-\beta\|_{L^2(\gamma_t)})\\
\nonumber&\le Ce^{-\mu t}+o(1)\|H\|_{L^2(\gamma_t)}
\end{align*}
by Lemma~\ref{lem:small}(ii), \eqref{equation:alphaX.torus}, \eqref{equation:harm.const}, and the area bound \eqref{equation:area.bound}. This proves (iii).

\emph{Case 3: $\kappa<0$.} Here, $\Lambda_t=\{0\}$ and $\bar\nu_t=0$ by definition. Hence, (i), (ii), and (iii) are straightforward. This completes the proof of Lemma~\ref{lem:spectral}.
\end{proof}

\subsection{Proof of Proposition~\ref{proposition:convergence} and Remark~\ref{remark:convergence}}

Now, combining \eqref{equation:H.L2.evolution}, \eqref{equation:gap}, and Lemma~\ref{lem:spectral}, we find
\[
\frac d{dt}\|H\|_{L^2(\gamma_t)}^2
\le(-2(\bar\nu_\infty-\kappa)+o(1))\|H\|_{L^2(\gamma_t)}^2+Ce^{-\mu t}\|H\|_{L^2(\gamma_t)}+Ce^{-2\mu t},
\]
where $\bar\nu_\infty-\kappa>0$. For any $\epsilon\in(0,\bar\nu_\infty-\kappa)$, we have
\[
Ce^{-\mu t}\|H\|_{L^2(\gamma_t)}
\le\epsilon\|H\|_{L^2(\gamma_t)}^2+C_\epsilon e^{-2\mu t}
\]
and therefore
\[
\frac d{dt}\|H\|_{L^2(\gamma_t)}^2
\le(-2(\bar\nu_\infty-\kappa)+2\epsilon)\|H\|_{L^2(\gamma_t)}^2+C_\epsilon e^{-2\mu t},
\]
for sufficiently large $t>0$. Consequently, $\|H\|_{L^2(\gamma_t)}^2\le C_{\mu_0}e^{-2\mu_0t}$ for any $0<\mu_0\le\mu$ with $\mu_0<\bar\nu_\infty-\kappa-\epsilon$. This completes the proof of Proposition~\ref{proposition:convergence}.

When $\kappa\ne0$, the previous proof already gives the optimal decay rate described in Remark~\ref{remark:convergence}. However, when $\kappa=0$, we have not yet obtained this optimal decay rate because $\bar\nu_\infty=\nu_*(\Gamma_\infty)$ is initially unknown. Fortunately, once we prove Proposition~\ref{proposition:convergence}, we obtain the smooth convergence in Corollary~\ref{corollary:convergence}. Then, we can run the previous argument again; this time we know $\bar\nu_\infty=\nu_*(\Gamma_\infty)$. This completes the proof of Remark~\ref{remark:convergence}.

\subsection{Proof of Theorem~\ref{theorem:main}}

By Corollary~\ref{corollary:convergence}, $\Gamma_t$ converges smoothly and exponentially to a minimal Lagrangian surface. It remains to verify the remaining properties of the limit $\Gamma_\infty$ stated in Theorem~\ref{theorem:main}. In the case $\kappa\ge0$, we know $\eta$ has a uniform positive lower bound \eqref{equation:eta.bound}, and so $\Gamma_\infty$ is the graph of an area-preserving map. Moreover, $\Gamma_\infty$ is totally geodesic if $\kappa\ge0$ by Proposition~\ref{proposition:A.bound}, and is the graph of an isometry if $\kappa>0$ by the fact that $\eta_t\to1$. Below, assume $\kappa<0$, in which case the lower bound \eqref{equation:eta.bound} on $\eta$ degenerates as $t\to\infty$ and a separate argument is needed.

Now, $\eta_t\to\eta_\infty=*(\omega''|_{\Gamma_\infty})\in[0,1]$ smoothly as $t\to\infty$, and
\[
\int_{\Gamma_t}\eta_t
=2\int_{\Gamma_t}\omega_t
=2\int_\Sigma\omega_t
=2\mathrm{area}_{g_t}(\Sigma)\to2\mathrm{area}_{g_\infty}(\Sigma),
\]
so
\begin{equation}
\label{equation:eta.L1}
\int_{\Gamma_\infty}\eta_\infty
=2\mathrm{area}_{g_\infty}(\Sigma)>0.
\end{equation}
Since $\Gamma_\infty\subset(M,G_\infty)$ is itself a stationary solution to the coupled flow \eqref{equation:RF} and \eqref{equation:MCF}, we see from \eqref{equation:eta.evolution} that
\[
-\Delta\eta_\infty
=2|A|^2\eta_\infty+\kappa\eta_\infty(1-\eta_\infty^2).
\]
By the Harnack inequality, either $\eta_\infty>0$ everywhere or $\eta_\infty=0$ identically. In the first case we are done, and the second case is impossible by \eqref{equation:eta.L1}. This completes the proof of Theorem~\ref{theorem:main}.

\end{document}